\documentclass[12pt]{amsart}
\usepackage{amsmath}
\usepackage[colorlinks=true, urlcolor=blue, linkcolor=red]{hyperref}

\usepackage{geometry}                
\usepackage{graphicx}
\usepackage{amssymb,amsthm,mathtools}
\usepackage{bbm}
\usepackage[capitalize]{cleveref}
\usepackage{epstopdf}
\usepackage{xcolor}
\usepackage{tikz}
\usepackage{nicematrix}
\newtheorem{theorem}{Theorem}[section] 

\theoremstyle{definition} 
\newtheorem{definition}[theorem]{Definition}
\newtheorem{remark}[theorem]{Remark}
\newtheorem{example}[theorem]{Example}

\title{Topological Obstructions to Shared Priors}

\author{Owen D. Biesel}
\address{Southern Connecticut State University}
\email{bieselo1@southernct.edu}

\author{Colin McSwiggen}
\address{Academia Sinica}
\email{csm@as.edu.tw}

\author{Ted Theodosopoulos}
\address{The Nueva School}
\email{ttheodosopoulos@nuevaschool.org}

\author{Michael G. Titelbaum}
\address{University of Wisconsin--Madison}
\email{titelbaum@wisc.edu}

\begin{document}

\begin{abstract}
        Given a finite collection of probability measures defined on subsets of a measurable space, how can we determine if they are compatible, in the sense that they can be realized as conditional distributions of a single probability measure on the full space?  This formulation of the consistency problem for conditional probabilities is significant in Bayesian epistemology and probabilistic reasoning, as it describes the conditions under which a collection of agents can reach agreement by sharing information.  We derive a necessary and sufficient condition under which joint compatibility is equivalent to pairwise compatibility.  This condition is stated in terms of the cohomology of a simplicial complex constructed from the given probability measures, exposing a novel application of algebraic topology to Bayesian reasoning.
\end{abstract}

\maketitle

\tableofcontents

\section{Introduction and motivation}

Which types of disagreements can be resolved through discussion?  That is, given a group of people holding different beliefs about the world, how can we determine whether their beliefs are compatible, in the sense that they might all agree if everyone had access to the same information? This question has provoked considerable debate among philosophers of knowledge and belief. It also contains surprising mathematical subtleties.  In this paper, we study a simple mathematical model of this problem and find that it reveals a novel link between probability and algebraic topology.

We consider a collection of agents, indexed by some finite set $I$, who seek to reconcile their beliefs about the world.  We work in a Bayesian setting where the possible states of the world are modeled by a measurable space $A$, and the beliefs of each agent $i \in I$ are modeled by a \emph{credence function}, which is a probability measure $P_i$ on some measurable subset $A_i \subseteq A$.  The set $A_i$ represents all possibilities of which the agent is aware; if two agents $i,j \in I$ have access to different information or differ in their conceptual understandings of the world, then we may have $A_i \ne A_j$.  We say that these two agents are \emph{compatible} if they agree on the conditional probabilities of the outcomes that they both consider: that is, if agents $i$ and $j$ both assign positive probability to the overlap $A_i \cap A_j$, then the conditional distributions of $P_i$ and $P_j$ on this overlap event are equal.

This notion of compatibility is equivalent to the existence of an \emph{ur-prior} for these two agents (Definition \ref{def:ur-prior} below), which is a probability measure on $A_i \cup A_j$ that recovers $P_i$ or $P_j$ when conditioned on $A_i$ or $A_j$ respectively. The definition of an ur-prior extends naturally to the case of more than two agents: an ur-prior for the collection of agents $I$ is a probability measure on the full space $A$ that coincides with each individual's credence function $P_i$ when conditioned on the event $A_i$, for all $i \in I$.  We say that these agents are \emph{compatible} if there exists a common ur-prior for all of their credence functions.  We then ask: what do we need to check in order to determine compatibility?  In particular, does the compatibility of every pair $i,j \in I$ imply the compatibility of $I$ as a whole?

As we will show, the answer to this latter question is \emph{no}, but with qualifications: in general, pairwise compatibility of individuals does not imply compatibility of the collective; however, in certain cases it does.  The main theorem of this paper is a necessary and sufficient condition that determines when these two criteria are indeed equivalent.  This condition depends on the agents' \emph{overlap simplicial complex} (Definition \ref{def:ovelap} below), a topological object constructed from the overlap events of different subsets of agents.  The following is an informal statement of this result, which is stated and proved rigorously as Theorem \ref{thm-cohomology-urprior} below.

\begin{theorem} \label{thm:main-intro}
    Let $\mathcal{X}$ be an overlap simplicial complex for a collection of agents.  The following are equivalent:
    \begin{itemize}
        \item For \emph{all} collections of agents with overlap simplicial complex $\mathcal{X}$, pairwise compatibility is equivalent to the existence of an ur-prior.
        \item The first cohomology of $\mathcal{X}$ vanishes.
    \end{itemize}
\end{theorem}

Cohomology is an invariant that can be defined for any topological space; we recall its construction, along with the other topological prerequisites, below in Section \ref{sec:cohomology-background}.

Theorem \ref{thm:main-intro} gives a criterion for a criterion: it says that to check compatibility of $I$, it is sufficient to check each pair $i,j \in I$, \emph{if and only if} $\mathcal{X}$ satisfies a certain topological condition.  This equivalence is interesting for several reasons.

First, it expresses a probabilistic fact (applicability of the pairwise compatibility test) in terms of a topological invariant (the cohomology of the overlap complex). As far as we are aware, this represents a novel link between probability and algebraic topology. The topological perspective offers clear explanations for some properties of ur-priors that may otherwise be counterintuitive, such as how compatibility is affected by the inclusion of a new agent (see Examples \ref{ex-plugged-hole} and \ref{ex:empty-fourway-intersection} below).

Second, it establishes a fully general, necessary and sufficient condition under which pairwise compatibility guarantees that an ur-prior is available. Because of the philosophical significance of ur-priors, philosophers have examined this problem and have proven some sufficient conditions \cite{Titelbaum2013, Titelbaum2022}. However, the problem of fully characterizing the relationship between these two criteria has been left open.  Our work gives a complete answer to this question.

Finally, from a computational point of view, Theorem \ref{thm:main-intro} gives a sufficient condition under which the existence of an ur-prior for a collection of agents can be determined in polynomial time.  Assuming that the sets $A_i$ are finite, pairwise compatibility of the agents can be checked in $O(|I|^2 \max_{i \in I} |A_i|)$ steps.  Moreover, the dimension of the first cohomology of a simplicial complex can be computed in a number of steps polynomial in the number of vertices \cite{EPbetti, Edelsbrunner2010}, and in the case of the overlap simplicial complex, the vertices correspond to elements of $I$. Thus it is possible to (1) determine whether the first cohomology of the overlap complex vanishes, and (2) if so, determine whether the agents have an ur-prior, both in polynomial time in the number of agents.

The paper is organized as follows.  In the remainder of this introduction, we give some background motivation for studying questions about the existence of ur-priors, and we review related literature in both mathematics and philosophy.  In the following section, we define some key concepts and present several illustrative examples that we will revisit throughout the paper.  Section \ref{sec:cohomology-background} gives a concise introduction to simplicial cohomology, Section \ref{sec-cohomology-urpriors} contains the statements and proofs of our main results, and Section \ref{sec:future-directions} outlines some directions for future work.

\subsection{Background and related literature}

The question of compatibility of conditional distributions has been studied from several perspectives in mathematics and statistics. Interest in the problem dates at least to the 1970's with the work of Besag on conditionally specified lattice models \cite{besag74}. In contrast to our approach, most existing work has focused on a setting in which we are given proposed distributions for a collection of $n$ random variables conditioned on certain subsets of the variables, and the task is to determine whether there exists a joint distribution with these conditionals.  This version of the problem has been treated in depth in the book by Arnold et al. \cite{arnold99}.  A few papers have also studied spaces of compatible families of conditional distributions as geometric objects in their own right \cite{matus03,morton13,SS06}. Notably, for finite probability spaces, Morton \cite{morton13} derived a complete set of algebraic conditions for the compatibility of conditional distributions given an arbitrary collection of events.  In the finite setting, our Theorem \ref{thm:main-intro} can be understood as determining when Morton's criteria can be reduced to the much simpler criterion of pairwise compatibility.  Models of opinion reconciliation, consensus formation and amalgamation of priors have also been the object of significant work in statistics, economics and multi-agent systems; see e.g. \cite{Acemoglu2011, Aumann1976, BlackwellDubins1962, GeanakoplosPolemarchakis1982, Grabisch2024, LehrerWagner1981}.

The problem that we study in this paper is also closely related to the \emph{marginal consistency problem}, which asks whether some given marginal distributions for subsets of a collection of random variables can be consistently ``glued'' to form a common joint distribution.  The marginal consistency problem has been widely studied in probability, statistics, and computer science, and the general case is known to be computationally intractible; see \cite{BenesStepan, Bhupatiraju, DO3way, KrizConditional, RoughgardenKearns, VejMarginals} and references therein. Although we will not say more about computational complexity in this paper, similar arguments likely imply that determining the existence of an ur-prior is NP-hard in general, so it is desirable to identify tractable special cases as discussed above.

Important further motivation for studying ur-priors comes from philosophy. Since the work of Kelly \cite{Kelly2005}, Feldman \cite{Feldman2006}, and Elga \cite{Elga2007}, a swath of the epistemology literature has focused on peer disagreement. One important question when we find people disagreeing is whether the conflict arises from differences in their information, or whether there is a more deep-seated difference in how they \emph{interpret} information.  One simple test is whether, were each party to share their total information with the other, the two would come to agree. Situated in the Bayesian formalism, this is our pairwise compatibility test.

While this assumption could be loosened, for present purposes we will assume that agents use traditional Bayesian conditionalization to update their credences in the face of new evidence.  Under that assumption, any differences in how agents interpret shared evidence must be traceable back to their starting points, or ``epistemic standards'' \cite{Schoenfield2014}.  Bayesian philosophers represent those standards with an ur-prior, a distribution over a set $A$ conditionalized on the agent's evidence $A_i$ at any given time to yield her credence distribution $P_i$ at that time. (This idea of representing an agent's interpretation of evidence with an ur-prior conditionalized on her total evidence dates back at least to Carnap \cite{Carnap1950}, with significant further development by Levi \cite{Levi1980}.) On this approach, a group of agents interpret evidence the same way if there exists an ur-prior common to all of them. Titelbaum \cite{Titelbaum2013} showed that \emph{if} such an ur-prior exists, the agents will be pairwise compatible, and Titelbaum \cite{Titelbaum2022} provided a sufficient condition for pairwise compatibility to guarantee an ur-prior. Here we present the first set of \emph{necessary and sufficient} conditions for pairwise compatibility to coincide with the existence of an ur-prior.

\section{Definitions and examples}

In the following, we consider a finite collection $I$ of agents, for which each agent $i\in I$ has a credence function $P_i$, a probability measure on a measurable subset $A_i$ of a measurable space $A$ common to all the agents. Given a probability measure $P$ on $A$ with $P(A_i) > 0$, we write $P|_{A_i}$ for its conditionalization on $A_i$, defined in the usual way via Bayes' rule: $P|_{A_i}(B) = P(B \cap A_i) / P(A_i)$ for measurable $B \subseteq A$. We are interested in the question of whether the agents could have started with a single credence function on all of $A$ and then arrived at their individual credences by conditionalizing on the evidence represented by the subsets $A_i$.  That is, we would like to know whether or not these agents have an \emph{ur-prior}.

\begin{definition}[Ur-prior]\label{def:ur-prior}
    An \emph{ur-prior} for the credence functions $(P_i)_{i \in I}$ is a probability measure $P$ on $\bigcup_{i\in I} A_i$ such that $P|_{A_i} = P_i$ for each $i\in I$.
\end{definition}

A necessary condition for having an ur-prior is that each pair of agents must agree on their overlaps:

\begin{definition}[Pairwise compatibility]\label{def:pairwise-compatibility}The agents are said to be \emph{pairwise compatible} if, whenever the credence functions of two agents $i$ and $j$ both assign positive probability to some event, their credence functions agree after conditionalizing on that event. Formally, this means that if $P_i(A_i\cap A_j)$ and $P_j(A_i \cap A_j)$ are both positive, then $P_i|_{A_i\cap A_j}$ and $P_j|_{A_i \cap A_j}$ are equal.
\end{definition}

\begin{example}\label{ex-urprior}
    Consider three agents attempting to determine the type of metal in a sample. (Any readers who are practicing metallurgists may need to suspend their disbelief.) They share a set \[A = \{\text{gold, platinum, aluminum, bismuth, silver, iron, copper}\}\] of possibilities for the type of metal, but each has eliminated some of the possibilities and has their own credence function on the remainder.

    Agent 1 has determined that the metal is gray, so has narrowed down the possibilities to the set $A_1 = \{\text{platinum, aluminum, silver, iron}\}.$ Agent 2 has noticed that the metal tarnishes, narrowing down the set to $A_2 = \{\text{bismuth, silver, iron, copper}\}$. And Agent 3 knows that the metal is precious, making the possibilities $A_3 = \{\text{gold, platinum, bismuth, silver}\}$. This arrangement of possibilities is shown in \cref{fig-urprior}.
    
    The agents' credences are represented by the probability mass functions below:
    \[\begin{array}{c|c|c|c|c}
    \text{metal} & \text{platinum} & \text{aluminum} & \text{silver} & \text{iron}\\ \hline
    P_1(\{\text{metal}\}) & 1/8 & 2/8 & 2/8 & 3/8
    \end{array}\]
    \[\begin{array}{c|c|c|c|c}
    \text{metal} & \text{bismuth} & \text{silver} & \text{iron} & \text{copper}\\ \hline
    P_2(\{\text{metal}\}) & 3/20 & 4/20 & 6/20 & 7/20
    \end{array}\]
    \[\begin{array}{c|c|c|c|c}
    \text{metal} & \text{gold} & \text{platinum} & \text{bismuth} & \text{silver}\\ \hline
    P_3(\{\text{metal}\}) & 1/10 & 2/10 & 3/10 & 4/10
    \end{array}\]

This collection of agents is pairwise compatible: for example, if Agents 1 and 2 were to pool their evidence, they would narrow down the possibilities to $A_{12} = \{\text{silver, iron}\}$, and if these same two agents were to conditionalize on that subset they would both assign probability $2/5$ to silver and $3/5$ to iron.

These agents also have an ur-prior: each credence function $P_i$ is the conditionalization to $A_i$ of the following probability mass function on $A$:
\[\begin{array}{c|c|c|c|c|c|c|c}
    \text{metal} & \text{gold} & \text{platinum} & \text{aluminum} & \text{bismuth} & \text{silver} & \text{iron} & \text{copper}\\ \hline
    P(\{\text{metal}\}) & 1/27 & 2/27 & 4/27 & 3/27 & 4/27 & 6/27 & 7/27
    \end{array}\]
\end{example}

It is the main result of \cite{biesel24} that, if a collection of agents is pairwise compatible \emph{and} there is an event to which they all assign positive probability (in this case, $\{\text{silver}\}$), then the agents must also have an ur-prior. In this special case where at least one event is accorded positive probability by all agents, the existence of a common ur-prior means that were they to share all their evidence and update by conditionalizing, their posterior credences would agree.

Pairwise compatibility, however, does not always guarantee the existence of an ur-prior. More specifically, the key criterion is whether there are any ``holes'' in the agents' \emph{overlap simplicial complex}:\footnote{For our purposes, a \emph{simplicial complex} is just a collection $\Delta$ of non-empty finite sets that are closed under taking subsets: if $X \in \Delta$ and $Y \subseteq X$, then $Y \in \Delta$.  We can then regard $\Delta$ as a combinatorial description of a collection of geometric simplices whose $k$-dimensional faces correspond to sets of cardinality $k$ in $\Delta$.

For readers who are aware of the \emph{nerve complex} of a collection of sets, the overlap simplicial complex is closely related, but the definition differs slightly due to the requirement that the agents assign positive probability to the intersections in order for the corresponding simplices to be included.}

\begin{definition}[Overlap simplicial complex]\label{def:ovelap} The \emph{overlap simplicial complex} on vertex set $I$ is defined to be the set $\mathcal{X}$ of those subsets $J\subseteq I$ such that for all $j\in J$, we have $P_j(\bigcap_{k\in J} A_k) > 0$. If $J\in\mathcal{X}$, we say that the agents in $J$ \emph{overlap}.
\end{definition}

\begin{figure}
    \[\begin{tikzpicture}
        \node at (0,0.3) {silver};
        \node at (0,2) {platinum};
        \node at (-2,1.2) {aluminum};
        \node at (-1.9,-0.8) {iron};
        \node at (0,-2) {copper};
        \node at (1.9,-0.8) {bismuth};
        \node at (2,1.2) {gold};
        \draw [rotate around={60:(-1.2,1)}] (-1.2,1) ellipse[x radius = 3, y radius = 2];
        \draw [dashed, rotate around={-60:(1.2,1)}] (1.2,1) ellipse[x radius = 3, y radius = 2];
        \draw[dotted, thick] (0,-1/1) ellipse[x radius = 3, y radius = 2];
        \node at (150:3.9) {$A_1$};
        \node at (250:3.3) {$A_2$};
        \node at (10:3.8) {$A_3$};
     \draw[fill=lightgray, thick] (5.27,1) -- (8.73,1) -- (7,-2) -- cycle;
    \filldraw[fill=white, line width=1] (5.27,1) circle (.25);
    \filldraw[fill=white, line width=1.5, dashed] (8.73,1) circle (.25);
    \filldraw[fill=white, line width=1.5, dotted] (7,-2) circle (.25);
     \node at (5.27,1) {1};
     \node at (8.73,1) {3};
     \node at (7,-2) {2};
    \end{tikzpicture}\]
    \caption{\label{fig-urprior}Left: Each of the three agents has narrowed down the mystery metal to four of the seven possibilities. Every pairwise compatible collection of credences that overlaps in this way has an ur-prior. Right: The overlap complex for these three agents is a filled-in triangle, since every subset of agents assigns nonzero probability to some event.}
\end{figure}
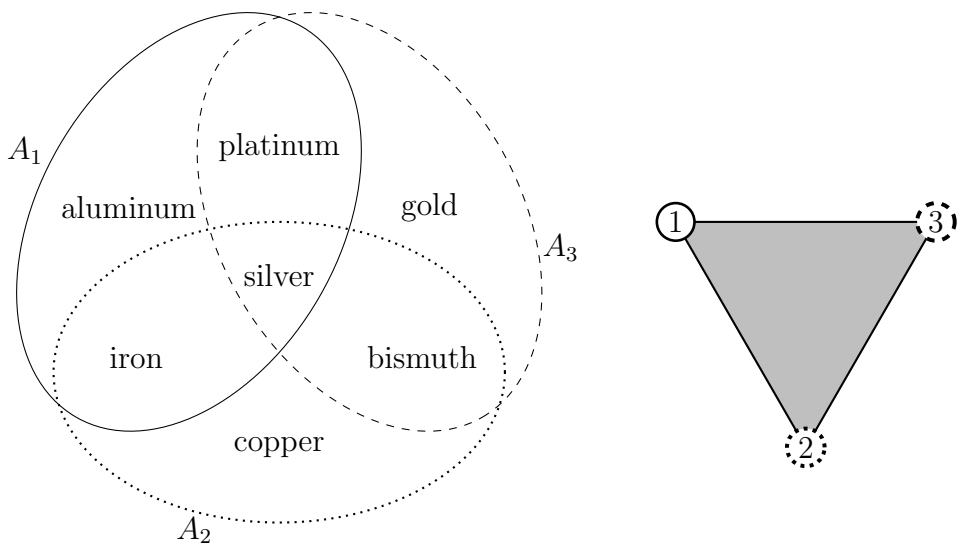

\begin{example}\label{ex-no-urprior}
    Let us consider a modified version of Example \ref{ex-urprior}, in which the agents do not consider silver a possibility, making the state space \[A = \{\text{gold, platinum, aluminum, bismuth, iron, copper}\},\] as shown in \cref{fig-no-urprior}. The following collection of credences is pairwise compatible but has no ur-prior. (Note that these are \emph{not} credences obtained via conditionalization from those in Example \ref{ex-urprior}.)

    \[\begin{array}{c|c|c|c}
    \text{metal} & \text{platinum} & \text{aluminum}  & \text{iron}\\ \hline
    P_1(\{\text{metal}\}) & 0.2 & 0.5 & 0.3
    \end{array}\]
    \[\begin{array}{c|c|c|c}
    \text{metal} & \text{bismuth} & \text{iron} & \text{copper}\\ \hline
    P_2(\{\text{metal}\}) & 0.3 & 0.2 & 0.5 
    \end{array}\]
    \[\begin{array}{c|c|c|c}
    \text{metal} & \text{gold} & \text{platinum} & \text{bismuth} \\ \hline
    P_3(\{\text{metal}\}) & 0.5 & 0.3 & 0.2
    \end{array}\]
In this case there is no ur-prior $P$, since conditionalizing to Agent 1's credence function would require $P(\{\text{iron}\})>P(\{\text{platinum}\})$, while Agent 2 requires $P(\{\text{bismuth}\})>P(\{\text{iron}\})$, and Agent 3 requires $P(\{\text{platinum}\})>P(\{\text{bismuth}\})$.
\end{example}

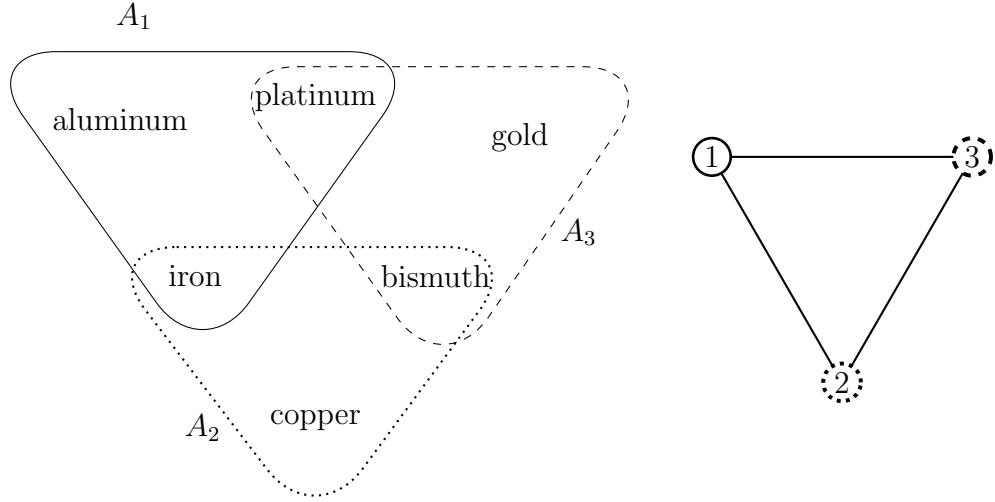
\begin{figure}
    \[\begin{tikzpicture}
        \node at (90:1.8) {platinum};
        \node at (150:3) {aluminum};
        \node at (200:1.7) {iron};
        \node at (270:2.5) {copper};
        \node at (340:1.7) {bismuth};
        \node at (25:3) {gold};
        \draw[rounded corners=30pt] (1.5,2.4) -- (-4.5,2.4) -- (-1.5,-1.8) -- cycle;
        \draw[dashed, rounded corners=30pt] (-1.3,2.2) -- (4.6,2.2) -- (1.7,-2) -- cycle;
        \draw[dotted, thick, rounded corners=30pt] (-2.9,-0.2) -- (2.8,-0.2) -- (0,-4) -- cycle;
        \node at (130:3.75) {$A_1$};
        \node at (240:3) {$A_2$};
        \node at (0:3.5) {$A_3$};
    \draw[fill=white, thick] (5.27,1) -- (8.73,1) -- (7,-2) -- cycle;
    \filldraw[fill=white, line width=1] (5.27,1) circle (.25);
    \filldraw[fill=white, line width=1.5, dashed] (8.73,1) circle (.25);
    \filldraw[fill=white, line width=1.5, dotted] (7,-2) circle (.25);
     \node at (5.27,1) {1};
     \node at (8.73,1) {3};
     \node at (7,-2) {2};
    \end{tikzpicture}\]
    \caption{\label{fig-no-urprior}Left: Now each agent has narrowed down the possibilities to three of the six. In this arrangement, the agents may be pairwise compatible but have no ur-prior. Right: The overlap complex for the three agents is an unfilled triangle, since each pair of agents assigns nonzero probability to some common event but the three agents together do not.}
\end{figure}

In a sense that we will make precise, the problem in passing from pairwise compatibility to the existence of an ur-prior is the ``hole'' in the center of \cref{fig-no-urprior}.

\begin{example}\label{ex-plugged-hole}
    We now consider a case in which the nonexistence of an ur-prior implies that the agents cannot be pairwise compatible. Suppose that a fourth agent joins the three in Example \ref{ex-no-urprior}, with evidence narrowing down the possibilities to $A_4 = \{\text{platinum}, \text{iron}, \text{bismuth}\}$, as shown in \cref{fig-plugged-hole}. This new agent ``plugs the hole'' in the overlap complex from Example \ref{ex-no-urprior}, so that the new overlap complex for all four agents has trivial cohomology. (For details of the cohomology calculation, see Example \ref{ex:coh-plugged-hole} below.) 
    
    These four agents cannot possibly have an ur-prior, since the original three agents in Example \ref{ex-no-urprior} did not. \cref{thm-cohomology-urprior} therefore implies that these agents cannot be pairwise compatible.  Since we have already checked that Agents 1, 2 and 3 \emph{are} pairwise compatible, we can conclude that there is no possible credence function for Agent 4 that would be compatible with each of the others.
    
To check this directly, write $P_4(\{\text{platinum}\}) = \alpha$ and $P_4(\{\text{iron}\}) = \beta$, so that $P_4(\{\text{bismuth}\}) = 1-\alpha-\beta$.  The requirement of pairwise compatibility between Agents $1$ \& $2$, $1$ \& $3$, and $2$ \& $3$ leads to the following three equations, respectively:
\begin{eqnarray}
\frac{2}{5} & = & \frac{\alpha}{\alpha+\beta} \label{eq:imp1} \\
\frac{2}{5} & = & \frac{\beta}{1-\alpha} \label{eq:imp2} \\
\frac{2}{5} & = & \frac{1-\alpha-\beta}{1-\beta} \label{eq:imp3}
\end{eqnarray}
where $0 \leq \alpha, \beta \leq 1$ and $\alpha+\beta \leq 1$.  Equations (\ref{eq:imp1}) and (\ref{eq:imp2}) imply that $\alpha = \frac{4}{19}$ and $\beta = \frac{6}{19}$, which make the right-hand side of equation (\ref{eq:imp3}) equal to $\frac{9}{13}$, a contradiction.
\end{example}

At this point, it may seem that pairwise compatibility is equivalent to the existence of an ur-prior if and only there is a subset to which all of the agents assign nonzero probability.  However, this is not the case, as the following final example illustrates.

\begin{example}\label{ex:empty-fourway-intersection}
    Suppose now that a fourth agent joins the three agents from Example \ref{ex-urprior}.  In particular, let the fourth agent's credence function be given by
$$\begin{array}{c|c|c|c}
    \text{metal} & \text{platinum} & \text{bismuth} & \text{iron} \\ \hline
    P_4(\{\text{metal}\}) & 2/11 & 3/11 & 6/11 
    \end{array}$$
One can readily check that this new agent is pairwise compatible with the previous three, who were themselves pairwise compatible with one another.  There is clearly no subset to which all four agents assign nonzero probability.  Nevertheless, they possess an ur-prior, given by
$$\begin{array}{c|c|c|c|c|c|c|c}
    \text{metal} & \text{platinum} & \text{bismuth} & \text{iron} & \text{copper} & \text{gold} & \text{aluminum} & \text{silver} \\ \hline
    P(\{\text{metal}\}) & 2/27 & 3/27 & 6/27 & 7/27 & 1/27 & 4/27 & 4/27
    \end{array}$$
For readers who are already familiar with cohomology, we note that it is only the first cohomology of the overlap simplicial complex that needs to vanish in order to guarantee the equivalence between pairwise compatibility and the existence of an ur-prior.  In this final example, the overlap simplicial complex has nontrivial second cohomology, but its first cohomology is trivial, which suffices.  
\end{example}

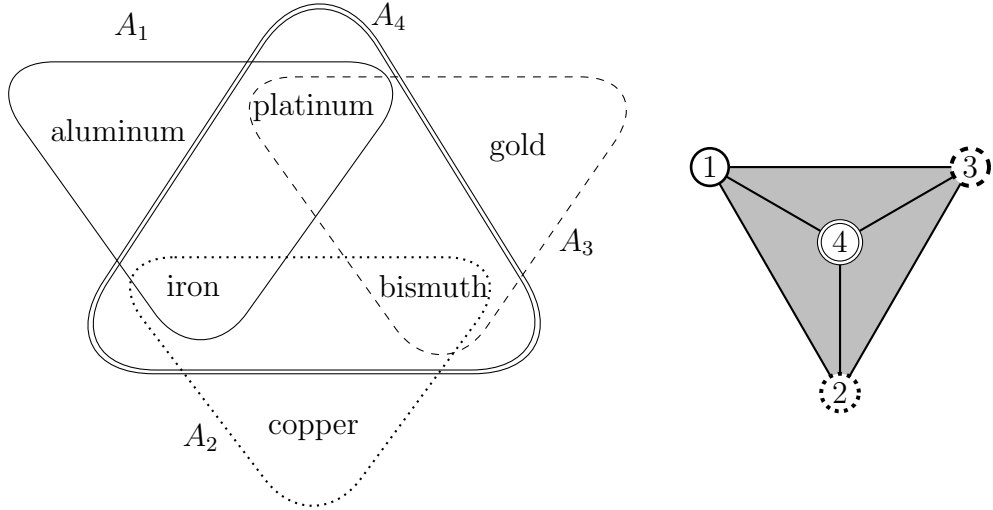
\begin{figure}
    \[\begin{tikzpicture}
        \node at (90:1.8) {platinum};
        \node at (150:3) {aluminum};
        \node at (200:1.7) {iron};
        \node at (270:2.5) {copper};
        \node at (340:1.7) {bismuth};
        \node at (25:3) {gold};
        \draw[rounded corners=30pt] (1.5,2.4) -- (-4.5,2.4) -- (-1.5,-1.8) -- cycle;
        \draw[dashed, rounded corners=30pt] (-1.3,2.2) -- (4.6,2.2) -- (1.7,-2) -- cycle;
        \draw[dotted, thick, rounded corners=30pt] (-2.9,-0.2) -- (2.8,-0.2) -- (0,-4) -- cycle;
        \draw[rounded corners=40pt] (0.1,3.8) -- (-3.5,-1.7) -- (3.5,-1.7) -- cycle;
        \draw[rounded corners=42pt] (0.1,3.9) -- (-3.6,-1.75) -- (3.6,-1.75) -- cycle;
        \node at (130:3.75) {$A_1$};
        \node at (240:3) {$A_2$};
        \node at (0:3.5) {$A_3$};
        \node at (1,3) {$A_4$};
    \draw[fill=lightgray, thick] (5.27,1) -- (8.73,1) -- (7,-2) -- cycle;
    \draw[thick] (5.27,1) -- (7,0) -- (8.73,1);
    \draw[thick] (7,0) -- (7,-2);
    \filldraw[fill=white, line width=1] (5.27,1) circle (.25);
    \filldraw[fill=white, line width=1.5, dashed] (8.73,1) circle (.25);
    \filldraw[fill=white, line width=1.5, dotted] (7,-2) circle (.25);
    \filldraw[fill=white] (7,0) circle (.3);
    \filldraw[fill=white] (7,0) circle (.25);
     \node at (5.27,1) {1};
     \node at (8.73,1) {3};
     \node at (7,-2) {2};
     \node at (7,0) {4};
    \end{tikzpicture}\]
    \caption{\label{fig-plugged-hole}Left: With the addition of a fourth agent in this configuration, pairwise compatibility again implies the existence of an ur-prior. Right: Since the fourth agent overlaps with each pair from the original three agents, the overlap complex is three filled triangles joined together without leaving any ``holes.''}
\end{figure}

To explain the precise meaning of ``holes,'' we turn to simplicial cohomology. Readers familiar with simplicial cohomology groups may skip to \cref{sec-cohomology-urpriors}.

\section{Introduction to simplicial cohomology}
\label{sec:cohomology-background}

\begin{definition}[First cohomology of a simplicial complex] \label{def:H1}
 Given a simplicial complex $\mathcal{X}$ on vertex set $I$, its \emph{first (real) cohomology} $H^1(\mathcal{X},\mathbb{R})$ is defined as follows. First, with an arbitrary order on the set $I$, we define the following sets:
 \begin{itemize}
  \item The set $\mathcal{X}_0=I$ is the set of vertices.
  \item The set $\mathcal{X}_1$ is the set of pairs of vertices $(i,j)$ such that $i<j$ and $\{i,j\}\in \mathcal{X}$.
  \item The set $\mathcal{X}_2$ is the set of triples of vertices $(i,j,k)$ such that $i<j<k$ and $\{i,j,k\}\in \mathcal{X}$.
 \end{itemize}
 A function $\mathcal{X}_k\to\mathbb{R}$ is called a \emph{$k$-cochain}, and the set of $k$-cochains is denoted $C^k$. There are functions $C^0\to C^1$ and $C^1\to C^2$, both denoted $\delta$, defined as follows: If $f:\mathcal{X}_0\to \mathbb{R}$ is a $0$-cochain, we define a $1$-cochain $\delta f : \mathcal{X}_1 \to \mathbb{R}$ by 
 \[\delta f(i, j) = f(j) - f(i),\]
 and if $g: \mathcal{X}_1 \to \mathbb{R}$ is a $1$-cochain, then we define a $2$-cochain $\delta g : \mathcal{X}_2 \to \mathbb{R}$ by 
 \[\delta g(i,j,k) = g(j,k) - g(i,k) + g(i,j).\]
 The 1-cochains $g$ such that $\delta g = 0$ are called \emph{cocycles}. The 1-cochains of the form $\delta f$ are called \emph{coboundaries}.  We can readily see that every coboundary is also a cocycle.  Specifically, let $g=\delta f \in C^1$ be a coboundary, where $f \in C^0$ is some 0-cochain.  Then $\delta g \in C^2$ such that for all $(i,j,k) \in \mathcal{X}_2$, 
 \begin{eqnarray*}
 \delta g (i,j,k) & = & \delta f (j,k) - \delta f (i,k) + \delta f (i,j) \\
 & = & \left[f(k) - f(j) \right] - \left[f(k) - f(i) \right] + \left[f(j) - f(i) \right] \\
 & = & f(k) - f(j) - f(k) + f(i) + f(j) - f(i) \\
 & = & 0
 \end{eqnarray*}
 and so $g$ is a cocycle.

 We define an equivalence relation on the set of cocycles, by saying that two cocycles are \emph{equivalent} if their difference is a coboundary. The set of equivalence classes of cocycles is called the \emph{(first, real) cohomology} of $\mathcal{X}$, denoted $H^1(\mathcal{X},\mathbb{R})$.
\end{definition}

\begin{remark}\label{rem-h1}
    There are two fundamentally different possibilities for the cohomology $H^1(\mathcal{X},\mathbb{R})$:
    \begin{enumerate}
        \item $H^1(\mathcal{X},\mathbb{R})$ has just one element, because all cocycles are already coboundaries. This means that, if $g: \mathcal{X}_1\to\mathbb{R}$ is any function satisfying $\delta g = 0$, then there must be a function $f: \mathcal{X}_0\to\mathbb{R}$ such that $g = \delta f$. In this case, we say that the cohomology \emph{vanishes} and write $H^1(\mathcal{X},\mathbb{R}) = 0$. This corresponds to the simplicial complex $\mathcal{X}$ having no ``holes.''
        \item $H^1(\mathcal{X},\mathbb{R})$ has more than one element. This means that there is a cocycle $g: \mathcal{X}_1 \to \mathbb{R}$ not equivalent to the zero function, so it is impossible to find a function $f: \mathcal{X}_0\to \mathbb{R}$ satisfying the equation $g = \delta f$. In this case, we say that the cohomology is nonzero, and the simplicial complex $\mathcal{X}$ has ``holes.''
    \end{enumerate}
\end{remark}

The next section contains the statements and proofs of our two main theorems, which establish a connection between the cohomology of the agents' overlap simplicial complex on one hand, and, on the other, the relationship between pairwise compatibility and the existence of an ur-prior for their credence functions.  It is worth pointing out that the notions linked by these theorems --- cohomology and pairwise compatibility/ur-priors --- involve distinct categories.  Specifically, cohomology is a topological property of the subsets of possibilities that the agents entertain (see Definitions \ref{def:ovelap} and \ref{def:H1}), while pairwise compatibility and the existence of an ur-prior are properties of the credence functions that the agents assign to those subsets (see Definitions \ref{def:ur-prior} and \ref{def:pairwise-compatibility}).  Thus, the latter depend on the credence functions themselves, while the former is a property of the collection of overlapping subsets and does not depend on a specific choice of credences.

Before moving on to the statement and proof of the general theorems, we return briefly to the cohomology computations for Examples \ref{ex-plugged-hole} and \ref{ex:empty-fourway-intersection}.  In particular, as seen in \cref{fig-plugged-hole} (Right), the fourth agent in Example \ref{ex-plugged-hole} has closed the ``hole'' that previously appeared in the overlap simplicial complex of Agents $1$, $2$ and $3$ shown in \cref{fig-no-urprior} (Right). 
 We now use elementary matrix techniques to calculate the cohomology in the presence or absence of a ``hole'' in the overlap complex.  A more detailed treatment of such calculations can be found in the introductory texts by Hatcher \cite{Hatcher2001} or Edelsbrunner and Harer \cite{Edelsbrunner2010}.

\begin{example}
    Consider the ``filled triangle'' overlap complex from \cref{fig-urprior}. Is its cohomology trivial or nontrivial? We will use elementary matrix methods to calculate the dimensions of the vector spaces of cocycles and coboundaries. Since the two dimensions turn out to be equal, we deduce that every cocycle is a coboundary and thus the cohomology is trivial.

    A 1-cochain is a function $\mathcal{X}_1\to\mathbb{R}$, which we represent as a vector with three components indexed by each of the three edges in the triangle. For example, we write the 1-cochain $f$ sending $(1,2)$ to $x$, $(1,3)$ to $y$, and $(2,3)$ to $z$ as
    \[f = \begin{pNiceMatrix}[first-col]
    _{12} & x\\
    _{13} & y\\
    _{23} & z
    \end{pNiceMatrix}.\]
    Such a 1-cochain is a cocycle if $\delta f= 0$, meaning that $f$ is in the kernel of the matrix representing $\delta:\mathcal{X}_1\to\mathcal{X}_2$:
    \[\begin{pNiceMatrix}[first-col,first-row]
     & _{12} & _{13} & _{23}\\
     _{123} & 1 & -1 & 1
     \CodeAfter
  \tikz \draw (1-1) circle (2mm);
    \end{pNiceMatrix}\]
    We compute the size of this kernel by putting the matrix into reduced row-echelon form (rref) and counting the number of non-pivot columns (columns without a row's leading 1, circled above). This matrix is already in rref, with one pivot column and two non-pivot columns, so the kernel is two-dimensional and there is a two-parameter family of cocycles. (Standard linear algebra techniques parameterize this family as $\{(y-z,y,z):y,z\in\mathbb{R}\}$.)

    Meanwhile, to find the coboundaries, we need the image of the linear transformation $\delta:\mathcal{X}_0\to\mathcal{X}_1$, represented by the following matrix (and its rref):
    \[\begin{pNiceMatrix}[first-col,first-row]
     & _1 & _2 & _3\\
     _{12} & -1 & 1 & 0\\
     _{13} & -1 & 0 & 1\\
     _{23} & 0 & -1 & 1
    \end{pNiceMatrix} \overset{\text{rref}}{\rightsquigarrow}
    \begin{pNiceMatrix}
        1 & 0 & -1\\
        0 & 1 & -1\\
        0 & 0 & 0
        \CodeAfter
        \tikz \draw (1-1) circle (2mm);
        \tikz \draw (2-2) circle (2mm);
    \end{pNiceMatrix}\]
    The dimension of the image is the number of pivot columns in the rref, so we again see that the collection of coboundaries is also a two-dimensional vector space. Since the vector space of coboundaries is contained in the vector space of cocycles and they have the same dimension, they must be equal sets of vectors. Therefore every cocycle is a coboundary and the cohomology is trivial.
\end{example}

\begin{example}
    Now we calculate the cohomology of the ``unfilled'' triangle in \cref{fig-no-urprior}. The coboundaries are unchanged, a two-dimensional family of vectors. But with $\mathcal{X}_2$ empty, the condition for a 1-cochain to be a cocycle is now trivial, so the set of cocycles is the entire three-dimensional family of 1-cochains. The cohomology, therefore, is $(3-2=1)$-dimensional and thus nontrivial; there is one ``hole.''
\end{example}

\begin{example}\label{ex:coh-plugged-hole}
    Next we calculate the cohomology of the overlap complex in \cref{fig-plugged-hole} to show that it is again trivial. This time, a 1-cochain has six components, one from each of the six edges in the complex. We can determine which are cocycles by taking the kernel of the matrix
    \[\begin{pNiceMatrix}[first-col,first-row]
        & _{12} & _{13} & _{14} & _{23} & _{24} & _{34}\\
        _{124} & 1 & 0 & -1 & 0 & 1 & 0\\
        _{134} & 0 & 1 & -1 & 0 & 0 & 1\\
        _{234} & 0 & 0 & 0 & 1 & -1 & 1
        \CodeAfter
        \tikz \draw (1-1) circle (2mm);
        \tikz \draw (2-2) circle (2mm);
        \tikz \draw (3-4) circle (2mm);
    \end{pNiceMatrix}.\]
    This matrix is again already in rref, with the three pivots circled, so there are three non-pivot columns and the vector space of cocyles is 3-dimensional.

    Meanwhile, the coboundaries are the vectors in the image of the matrix
    \[\begin{pNiceMatrix}[first-col,first-row]
     & _1 & _2 & _3 & _4\\
     _{12} & -1 & 1 & 0 & 0\\
     _{13} & -1 & 0 & 1 & 0\\
     _{14} & -1 & 0 & 0 & 1\\
     _{23} & 0 & -1 & 1 & 0\\
     _{24} & 0 & -1 & 0 & 1\\
     _{34} & 0 & 0 & -1 & 1
    \end{pNiceMatrix} \overset{\text{rref}}{\rightsquigarrow}
    \begin{pNiceMatrix}
        1 & 0 & 0 & -1\\
        0 & 1 & 0 & -1\\
        0 & 0 & 1 & -1\\
        0 & 0 & 0 & 0 \\
        0 & 0 & 0 & 0 \\
        0 & 0 & 0 & 0 
        \CodeAfter
        \tikz \draw (1-1) circle (2mm);
        \tikz \draw (2-2) circle (2mm);
        \tikz \draw (3-3) circle (2mm);
    \end{pNiceMatrix}.\]
    The rref has three pivots, so the image is 3-dimensional. Therefore the vector spaces of coboundaries and cocycles are both 3-dimensional, so every cocycle is again a coboundary and the cohomology is trivial.
\end{example}

\begin{example}
    In Example \ref{ex:empty-fourway-intersection}, we claimed that the first cohomology of the overlap complex is trivial but the second cohomology is nontrivial. We verify both of these claims now. The vector space of 1-cochains that are cocycles is the kernel of
        \[\begin{pNiceMatrix}[first-col,first-row]
        & _{12} & _{13} & _{14} & _{23} & _{24} & _{34}\\
        _{123} & 1 & -1 & 0 & 1 & 0 & 0\\
        _{124} & 1 & 0 & -1 & 0 & 1 & 0\\
        _{134} & 0 & 1 & -1 & 0 & 0 & 1\\
        _{234} & 0 & 0 & 0 & 1 & -1 & 1
    \end{pNiceMatrix}\overset{\text{rref}}{\rightsquigarrow}
    \begin{pNiceMatrix}
         1 & 0 & -1 & 0 & 1 & 0\\
         0 & 1 & -1 & 0 & 0 & 1\\
         0 & 0 & 0 & 1 & -1 & 1\\
         0 & 0 & 0 & 0 & 0 & 0
         \CodeAfter
        \tikz \draw (1-1) circle (2mm);
        \tikz \draw (2-2) circle (2mm);
        \tikz \draw (3-4) circle (2mm);
    \end{pNiceMatrix},\]
    so the space of 1-cocycles is three-dimensional. Meanwhile, the space of 1-coboundaries is unchanged from Example \ref{ex:coh-plugged-hole} and is also three-dimensional, so all the 1-cocycles are 1-coboundaries and the first cohomology is trivial.

    For those who are aware of higher cohomology, the second cohomology is nontrivial: the entire four-dimensional space of 2-cochains are are cocycles because $\mathcal{X}_3$ is empty, but the space of 2-coboundaries is the image of the above matrix and is only three-dimensional, making the second cohomology one-dimensional. The second and higher cohomology groups turn out to have no bearing on the problems considered in this paper, as \cref{thm-cohomology-urprior} shows.
\end{example}

\section{Cohomology and ur-priors}\label{sec-cohomology-urpriors}

In this section we state and prove \cref{thm-cohomology-urprior}, our main result on cohomology and the existence of ur-priors for pairwise compatible credences.

We start by showing an auxiliary result that we will use in the proof of the main theorem.  Fix a finite set $I$ of agents, with credence functions $P_i$ on subsets $A_i$ of a shared measurable space $A$, and let $\mathcal{X}$ be their overlap simplicial complex. We use the notational convention that intersections of the sets $A_i$ are written with multiple subscripts, i.e.\ $A_i\cap A_j$ is abbreviated $A_{ij}$ and $A_i\cap A_j \cap A_k$ as $A_{ijk}$. Recall that if $(i,j)\in \mathcal{X}_1$, then $P_i(A_{ij})>0$ and $P_j(A_{ij})>0$.  We define a function $g: \mathcal{X}_1\to \mathbb{R}$ by
    \[g(i,j) \coloneqq \log\left(\frac{P_i(A_{ij})}{P_j(A_{ij})}\right).\]
(Any choice of base for the logarithm will work; we will use base $2$ for definiteness.)

\begin{theorem}\label{thm-coboundary}
 If the agents are pairwise compatible, then the function $g:\mathcal{X}_1\to \mathbb{R}$ is a cocycle. Furthermore, $g$ is a coboundary if and only if the agents' credences have an ur-prior.
\end{theorem}

\begin{proof}
    Note that the pairwise compatibility of the agents tells us that if $B\subseteq A_i\cap A_j$ is any other measurable set such that $P_i(B)$ and $P_j(B)$ are positive, then we also have $g(i,j) = \log(P_i(B)/P_j(B))$. This is because
    \begin{align*}
        \log\left(\frac{P_i(B)}{P_j(B)}\right) - \log\left(\frac{P_i(A_{ij})}{P_j(A_{ij})}\right) &= \log\left(\frac{P_i(B)/P_i(A_{ij})}{P_j(B)/P_j(A_{ij})}\right)\\
        &= \log\left(\frac{P_i|_{A_{ij}}(B)}{P_j|_{A_{ij}}(B)}\right) = \log(1) = 0
    \end{align*}
    since $P_i|_{A_{ij}} = P_j|_{A_{ij}}$ by pairwise compatibility.
    
    Next we show that $\delta g = 0$, i.e.\ that if $(i,j,k)\in\mathcal{X}_2$, then $g(j,k) - g(i,k) + g(i,j) = 0$. The assumption that $(i,j,k)\in\mathcal{X}_2$ means that agents $i$, $j$ and $k$ all assign positive probability to the event $A_i\cap A_j\cap A_k$. Therefore, we have
    \begin{align*}
        g(j,k) - g(i,k) + g(i,j) &= \log\left(\frac{P_j(A_{ijk})}{P_k(A_{ijk})}\right) - \log\left(\frac{P_i(A_{ijk})}{P_k(A_{ijk})}\right) + \log\left(\frac{P_i(A_{ijk})}{P_j(A_{ijk})}\right)\\
        &= \log\bigl(P_j(A_{ijk})\bigr)-\log\bigl(P_k(A_{ijk})\bigr)\\
        &\quad - \log\bigl(P_i(A_{ijk})\bigr) + \log\bigl(P_k(A_{ijk})\bigr)\\
        &\quad + \log\bigl(P_i(A_{ijk})\bigr) - \log\bigl(P_j(A_{ijk})\bigr)\\
        &= 0,
    \end{align*}
    so $g$ is a cocycle.

    Now we show that the cocycle $g$ is a coboundary if and only if the agents' credences have an ur-prior. Suppose first that an ur-prior exists: a credence function $P$ on the total measurable space $A$, such that for each agent $i$ we have $P_i = P|_{A_i}$. We use $P$ to define a function $f:\mathcal{X}_0\to \mathbb{R}$ by
    \[f(i) \coloneqq -\log(P(A_i)).\] We will show that $g$ is a coboundary by proving that $g = \delta f$, i.e.\ that $g(i,j) = f(i) - f(j)$.

    In particular, suppose that $(i,j)\in\mathcal{X}_1$ and consider
    \begin{align*}
        g(i,j) &= \log\left(\frac{P_i(A_{ij})}{P_j(A_{ij})}\right)\\
        &= \log\left(\frac{P|_{A_i}(A_{ij})}{P|_{A_j}(A_{ij})}\right)\\
        &= \log\left(\frac{P(A_{ij})/P(A_i)}{P(A_{ij})/P(A_j)}\right)\\
        &= -\log(P(A_i)) + \log(P(A_j)) = f(i) - f(j).
    \end{align*}

    Now for the other direction, suppose only that $g$ is a coboundary; we will prove that the agents have an ur-prior. Let $f:\mathcal{X}_0\to\mathbb{R}$ be a function such that $g(i,j) = f(i) - f(j)$ for all $(i,j)\in\mathcal{X}_1$. We use the $f(i)$ to scale each probability measure $P_i$ to a new measure $M_i$ on $A_i$ given by
    \[M_i(B) \coloneqq P_i(B) / \exp(f(i)).\]
    Observe that $M_i(A_{ij}) = M_j(A_{ij})$ for all $(i,j)\in\mathcal{X}_1$, because 
    \begin{align*}
        \frac{M_i(A_{ij})}{M_j(A_{ij})} &= \frac{P_i(A_{ij})/\exp(f(i))}{P_j(A_{ij})/\exp(f(j))}\\
        &= \exp(g(i,j) - f(i) + f(j))\\
        &= \exp(0) = 1.
    \end{align*}
    This also means that for all measurable $B\subseteq A_{ij}$, we have $M_i(B) = M_j(B)$, since
    \begin{align*}
        M_i(B) &= \frac{M_i(B)}{M_i(A_{ij})}M_i(A_{ij})\\
        &= \frac{P_i(B)}{P_i(A_{ij})}M_i(A_{ij})\\
        &= \frac{P_j(B)}{P_j(A_{ij})}M_j(A_{ij}) = M_j(B),
    \end{align*}
    using the pairwise compatibility $P_i|_{A_{ij}} = P_j|_{A_{ij}}$.
    Thus these measures $M_i$ on $A_i$ agree on their overlaps, and so glue to a single measure on all of $\bigcup_{i\in I} A_i$. Extending this measure by $0$ if necessary, we obtain a measure $M$ on all of $A$ with the property that for each $B\subset A_i$, we have $M(B) = M_i(B)$. Furthermore, since the set $I$ of agents is finite, we have $M(A)<\infty$, so we can normalize $M$ to obtain a probability measure $P$ on $A$ defined by $P(B) = M(B)/M(A)$. Then this $P$ is an ur-prior for the $P_i$, since for an arbitrary measurable $B\subseteq A_i$ we have
    \[ P|_{A_i}(B) = \frac{P(B)}{P(A_i)} = \frac{M(B)/M(A)}{M(A_i)/M(A)}= \frac{M_i(B)}{M_i(A_i)}= \frac{P_i(B) / \exp(f(i))}{P_i(A_i)/\exp(f(i))} = P_i(B),
    \]
    so $P|_{A_i} = P_i$ as desired.
\end{proof}

\begin{theorem}\label{thm-cohomology-urprior}
    Let $\mathcal{X}$ be a finite simplicial complex with vertex set $I$.
    \begin{enumerate}
        \item If $H^1(\mathcal{X},\mathbb{R})$ vanishes, then for any set of agents indexed by $I$ with overlap simplicial complex $\mathcal{X}$, having an ur-prior is equivalent to being pairwise compatible.
        \item If $H^1(\mathcal{X},\mathbb{R})\ne 0$, then there exists a set of pairwise compatible agents indexed by $I$ with overlap simplicial complex $\mathcal{X}$ but no ur-prior.
    \end{enumerate}
\end{theorem}

\begin{proof}
    The first part of the theorem is a straightforward corollary of \cref{thm-coboundary}. For one direction of implication, note that the existence of an ur-prior always implies that the agents are pairwise compatible. For the other direction, suppose that the agents are pairwise compatible and that $H^1(\mathcal{X},\mathbb{R})$ vanishes. Then because the agents are pairwise compatible, the function $g:\mathcal{X}_1\to\mathbb{R}$ of \cref{thm-coboundary} is a cocycle. On the other hand, the assumption that $H^1(\mathcal{X},\mathbb{R})$ vanishes means that every cocycle is also a coboundary (see Remark \ref{rem-h1}).  Thus $g$ is a coboundary, and so \cref{thm-coboundary} implies that the agents must have an ur-prior.

    To prove the second part of the theorem, we will use an arbitrary nonzero element of $H^1(\mathcal{X},\mathbb{R})$ to construct a collection of pairwise compatible agents with overlap complex $\mathcal{X}$ and no ur-prior.  Suppose now that $H^1(\mathcal{X},\mathbb{R})\ne 0$ and let $f: \mathcal{X}_1\to\mathbb{R}$ be a cocycle that is not a coboundary. First we will use $\mathcal{X}$ to build an appropriate measurable space $A$ with subsets $\{A_i: i\in I\}$ that overlap according to $\mathcal{X}$. Then we will use the cocycle $f$ to build measures on the sets $A_i$ that are related on overlaps by rescaling according to $f$. Finally, we will normalize the measures to a set of pairwise compatible probability measures and check that the resulting cocycle $g$ from \cref{thm-coboundary} is not a coboundary, implying that there is no possible ur-prior for this collection of agents.

        The measurable space of outcomes we will use is $A = \mathcal{X}$; that is, each simplex in the simplicial complex $\mathcal{X}$ will be its own point in the measurable space. The $\sigma$-algebra of measurable subsets is the full power set of $A$. The subset $A_i\subseteq A$ corresponding to each agent $i\in I$ is given by
        \[A_i = \{x\in \mathcal{X} : i\in x\}.\]
        For instance, $A_i$ contains as elements the singleton $\{i\}$, as well as any two-elements sets $\{i,j\}$ for which an edge between $i$ and $j$ exists in the simplicial complex, as well as any three-element subsets $\{i,j,k\}$ corresponding to triangles in $\mathcal{X}$ containing $i$, and so on. Note that $\bigcap_{i\in J}A_i = \{x\in \mathcal{X} : J\subseteq x\}$ is nonempty if and only if $J\in\mathcal{X}$, so to ensure that the overlap complex is $\mathcal{X}$ again we must make sure that agent $i$ assigns nonzero probability to \emph{every} point of $A_i$.

        Now we use the cocycle $f$ to build a measure $\mu_i$ on each $A_i$. Since each $A_i$ is finite, we need only assign a point mass to each element $x\in A_i$. Using the same (arbitrary) order on $I$ as in the definition of $H^1(\mathcal{X},\mathbb{R})$, we define the measure $\mu_i$ on $A_i$ by
        \[\mu_i(\{x\}) = \exp(f(i,k_0))\text{ where } k_0 = \max\{k\in x\}.\]
        Note that this means, if $x\in A_{ij}$, then 
        \begin{align*}
            \mu_i(\{x\})/\mu_j(\{x\}) &= \exp(f(i,k_0))/\exp(f(j,k_0))\\
            &= \exp(f(i,k_0) - f(j,k_0))\\
            &= \exp(f(i,j))
        \end{align*}
        using the cocycle condition on $f$ and the fact that $\{i,j,k_0\}$ is a triangle in $\mathcal{X}$. We therefore have $\mu_i(B)/\mu_j(B) = \exp(f(i,j))$ for all nonempty $B\subset A_{ij}$ as well.

        Now, since $A_i$ is finite and nonempty and each point of $A_i$ has positive mass under $\mu_i$, we can normalize $\mu_i$ to a probability measure $P_i$ by setting $P_i(B)\coloneqq \mu_i(B)/\mu_i(A_i)$ for each $B\subseteq A_i$. Considering these measures $P_i$ as credence functions for a collection of agents, we find that they are pairwise compatible: if $A_i$ and $A_j$ overlap then we have, for $B\subseteq A_{ij}$,
        \begin{align*}
            P_i|_{A_{ij}}(B) &= \frac{P_i(B)}{P_i(A_{ij})} = \frac{\mu_i(B)}{\mu_i(A_{ij})} = \frac{\mu_j(B)/\exp(f(i,j))}{\mu_j(A_{ij})/\exp(f(i,j))}\\
            &= \frac{\mu_j(B)}{\mu_j(A_{ij})} = \frac{P_j(B)}{P_j(A_{ij})} = P_j|_{A_{ij}}(B).
        \end{align*}
        Now we show that the cocycle $g$ corresponding to these agents is not a coboundary: it is in the same equivalence class as $f$. Indeed, we can calculate, for each $\{i,j\}\in \mathcal{X}$,
        \begin{align*}
            g(i,j) &\coloneqq \log\left(\frac{P_i(A_{ij})}{P_j(A_{ij})}\right)\\
            &= \log\left(\frac{\mu_i(A_{ij})/\mu_i(A_i)}{\mu_j(A_{ij})/\mu_j(A_j)}\right)\\
            &= \log\bigl(\exp(f(i,j))\bigr) - \log(\mu_i(A_i)) + \log(\mu_j(A_j))\\
            &= f(i,j) + \delta\alpha(i,j)
        \end{align*}
        where $\alpha:\mathcal{X}_0\to\mathbb{R}$ is defined by $\alpha(i)\coloneqq \log(\mu_i(A_i))$. Therefore $g$ differs from $f$ by a coboundary, and since $f$ itself is not a coboundary, neither is $g$.\qedhere
\end{proof}

\section{Directions for future work}
\label{sec:future-directions}

Our discussion so far has considered a simple model of agents, beliefs, agreement, and updating on evidence.  In closing, we point out two natural avenues for generalizing this work by relaxing assumptions.

The first assumption is that the agents share a common $\sigma$-algebra. That is, each agent's credence function is concentrated on a measurable subset $A_i$ of a common probability space $(A,\Sigma)$, inheriting the resulting sub-$\sigma$-algebra $\Sigma_i \coloneqq \sigma \left(\{B \subseteq A_i : B \in \Sigma\}\right)$.  What if the agents instead come with their own arbitrary $\sigma$-algebras on the sets $A_i$, which may not have any particular relation to one another?  This broader formalism would include philosophically relevant situations in which the agents could be modeled as possessing different, perhaps only partially translatable languages for describing the world.

In this setting of multiple distinct $\sigma$-algebras, which events should we consider in order to determine whether a collection of agents hold compatible beliefs?  One possibility might be to consider only sets that belong to the intersection $\sigma$-algebra $\bigcap_{i \in J} \Sigma_i$, where $J \subseteq I$ is the subset of agents under consideration. Another might be to consider instead the $\sigma$-algebra $\sigma(\bigcup_{i \in J} \Sigma_i)$ generated by all of the $\sigma$-algebras $\Sigma_i$, $i \in J$. The former is coarser; in our linguistic analogy, we could imagine that it describes a situation in which the agents are allowed only to communicate using words they all understand.  The latter, by contrast, is finer than the original individual $\sigma$-algebras, allowing the agents to combine their respective languages and make joint sense of new concepts.

In addition to considering multiple distinct $\sigma$-algebras, we could also consider different methods for updating the agents' credence functions. This brings us to the second major assumption that we have made in this paper, which is that all of the agents use the same update rule, namely Bayesian conditionalization.  What if different agents might apply different update rules?  Loosening this assumption would allow for scenarios in which two agents starting with the same prior credence function on the same $\sigma$-algebra might end up with incompatible posterior credences even when confronted with the same evidence.

In summary, there are many possibilities for generalizing the notion of agreement among agents that we have considered in this paper; these generalizations include models of agents with different background $\sigma$-algebras and different learning rules.  Cohomological methods like those developed in this paper may offer a versatile framework for understanding the possibilities and obstacles for reconciling beliefs in such models.

\section*{Acknowledgments}
The work of C.M. is partially supported by the National Science Foundation under grant number DMS-2103170, by the National Science and Technology Council of Taiwan under grant number 113WIA0110762, and by a Simons Investigator award via Sylvia Serfaty.

The work of M.T. is supported by the University of Wisconsin--Madison Office of
the Vice Chancellor for Research and Graduate Education with funding from the Wisconsin Alumni Research Foundation.


\begin{thebibliography}{99}

\bibitem{Acemoglu2011} Acemo\v{g}lu, Daron, et al. ``Bayesian learning in social networks.'' \emph{Review of Economic Studies} 78 (2011), 1201--1236.

\bibitem{arnold99} Arnold, Barry C., Enrique Castillo and Jos\'e Mar\'ia Sarabia. \emph{Conditional Specification of Statistical Models}. Springer-Verlag, 1999.

\bibitem{Aumann1976}
Aumann, Robert J. ``Agreeing to disagree.'' \emph{The Annals of Statistics} 4 (1976), 1236--1239.

\bibitem{BenesStepan} Bene\v{s}, Viktor and Josef \v{S}t\v{e}p\'an (eds.). \emph{Distributions with Given Marginals and Moment Problems}. Springer, 1997.

\bibitem{besag74} Besag, Julian. ``Spatial interaction and the statistical analysis of lattice systems.'' \emph{Journal of the Royal Statistical Society, Series B (Methodological)} 36 (1974): 192--236.

\bibitem{biesel24} Biesel, Owen D. ``Sheaves of probability.'' Preprint, 2024. \url{https://arxiv.org/abs/2401.01968}.

\bibitem{Bhupatiraju} Bhupatiraju, Surya. ``On the complexity of the marginal satisfiability problem.'' Preprint, 2012. \url{https://math.mit.edu/research/highschool/primes/materials/2012/Bhupatiraju.pdf}. Accessed 3 Oct. 2025.

\bibitem{BlackwellDubins1962}
Blackwell, David and Lester Dubins. ``Merging of opinions with increasing information.'' \emph{The Annals of Mathematical Statistics} 33 (1962), 882--886.


\bibitem{Carnap1950} Carnap, Rudolf. \emph{Logical Foundations of Probability.} University of Chicago Press, 1950.

\bibitem{DO3way} De Loera, Jesus and Shmuel Onn. ``The complexity of three-way statistical tables.''  \emph{SIAM Journal on Computing} 33 (2004):  819--836.

\bibitem{EPbetti} Edelsbrunner, Herbert and Salman Parsa. ``On the computational complexity of Betti numbers: reductions from matrix rank.'' In \emph{Proceedings of the Twenty-Fifth Annual ACM-SIAM Symposium on Discrete Algorithms}. Society for Industrial and Applied Mathematics, 2014: 152--160.

\bibitem{Edelsbrunner2010} Edelsbrunner, Herbert, and John Harer. \emph{Computational Topology: An Introduction}. American Mathematical Society, 2010.

\bibitem{Elga2007} Elga, Adam. ``Reflection and disagreement.'' \emph{No\^us} 41.3 (2007): 478--502.

\bibitem{Feldman2006} Feldman, Richard. ``Reasonable religious disagreements.'' In Louise Antony (ed.), \emph{Philosophers without gods: Meditations on atheism and the secular life}. Oxford University Press, 2006: 194--214.

\bibitem{GeanakoplosPolemarchakis1982}
Geanakoplos, John and Heraklis M. Polemarchakis. ``We can't disagree forever.'' \emph{Journal of Economic Theory} 28 (1982), 192--200.

\bibitem{Hatcher2001} Hatcher, Allen. \emph{Algebraic Topology.} Cambridge University Press, 2001.

\bibitem{Grabisch2024} Grabisch, Michel, and Yasar, M. Alperen. ``Frequentist belief update under ambiguous evidence in social networks.'' \emph{International Journal of Approximate Reasoning} 172 (2024), 109240.

\bibitem{Kelly2005} Kelly, Thomas. ``The epistemic significance of disagreement.'' \emph{Oxford Studies in Epistemology} 1.1 (2005): 167--196.

\bibitem{KrizConditional} K\v{r}\'i\v{z}, Otakar. ``Conditional problem for objective probability.'' \emph{Kybernetika} 34 (1998), 27--40.

\bibitem{LehrerWagner1981}
Lehrer, Keith and Carl Wagner. \emph{Rational Consensus in Science and Society: A Philosophical and Mathematical Study.} Springer Dordrecht, 1981.

\bibitem{Levi1980} Levi, Isaac. \emph{The Enterprise of Knowledge.} MIT Press, 1980.

\bibitem{matus03} Mat\'u\v{s}, Franti\v{s}ek. ``Conditional probabilities and permutahedron.'' \emph{Annales de l'Institut Henri Poincar\'e, Probabilit\'es et Statistiques} 39 (2003), 687--701.

\bibitem{morton13} Morton, Jason. ``Relations among conditional probabilities.'' \emph{Journal of Symbolic Computation} 50 (2013), 478--492.

\bibitem{RoughgardenKearns} Roughgarden, Tim and Michael Kearns. ``Marginals-to-models reducibility.''  In C.J. Burges et al. (eds.), \emph{Advances in Neural Information Processing Systems (NIPS 2013)}. Curran Associates, Inc., 2013. 

\bibitem{Schoenfield2014} Schoenfield, Miriam. ``Permission to believe: Why permissivism is true and what it tells us about ireelevant influences on belief.'' \emph{No\^us} 48.2 (2014): 193--218.

\bibitem{SS06} Slavkovic, Aleksandra B. and Seth Sullivant. ``The space of compatible full conditionals is a unimodular toric variety.'' \emph{Journal of Symbolic Computation} 41 (2006), 196--209.

\bibitem{Titelbaum2013} Titelbaum, Michael G. \emph{Quitting Certainties: A Bayesian Framework Modeling Degrees of Belief.} Oxford University Press, 2013.

\bibitem{Titelbaum2022} Titelbaum, Michael G. \emph{Fundamentals of Bayesian Epistemology.} Oxford University Press, 2022.


\bibitem{VejMarginals} Vejnarov\'a, Ji\v{r}ina. ``On marginal problem in evidence theory.'' In Fabio Cuzzolin (ed.), \emph{Belief Functions: Theory and Applications (BELIEF 2014).} Springer, 2014: 331--338.

\end{thebibliography}
\end{document}